\documentclass[british,11pt]{amsart}
\pdfoutput=1

\usepackage{mathtools}
\usepackage{microtype}
\usepackage[T1]{fontenc}
\usepackage[utf8]{inputenc}
\usepackage{babel}
\usepackage{csquotes}

\calclayout

\usepackage[backend=biber,backref=true,doi=false,url=false,isbn=false,date=year,
giveninits=true,maxbibnames=99,sortcites,hyperref=true]{biblatex}
\renewbibmacro{in:}{%
  \ifentrytype{article}{}{\printtext{\bibstring{in}\intitlepunct}}}
\DeclareFieldFormat[article,periodical]{volume}{\mkbibbold{#1}}%
\DeclareFieldFormat[article,periodical,unpublished]{citetitle}{#1}
\DeclareFieldFormat[article,periodical,unpublished]{title}{#1}
\DeclareFieldFormat{pages}{#1}
\DeclareFieldFormat[inproceedings]{title}{#1\isdot}
\DeclareFieldFormat[misc]{title}{#1\isdot}
\DeclareFieldFormat[misc,article]{note}{#1\addcomma}

\AtEveryBibitem{
  \clearfield{number}}
\makeatletter
\@namedef{subjclassname@2020}{
  \textup{2020} Mathematics Subject Classification}
\makeatother

\newtheorem{theorem}{Theorem}
\newtheorem{lemma}[theorem]{Lemma}
\newtheorem{corollary}[theorem]{Corollary}
 \newtheorem{prop}[theorem]{Proposition}

\theoremstyle{definition}
\newtheorem{definition}[theorem]{Definition}

\begin{document}

\title{Metric Spaces where Geodesics are Never Unique}%
\author{Amlan Banaji}
\date{}



\begin{abstract}
This article concerns a class of metric spaces, which we call \emph{multigeodesic spaces}, where between any two distinct points there exist multiple distinct minimising geodesics. We provide a simple characterisation of multigeodesic normed spaces and deduce that $(C([0,1]),||\cdot||_1)$ is an example of such a space, but that finite-dimensional normed spaces are not. We also investigate what additional features are possible in arbitrary metric spaces which are multigeodesic. 
\end{abstract}

\maketitle

\section{Introduction}
Geodesics are objects of great significance in geometry, with particular importance in Riemannian geometry and general relativity, as well as in the calculus of variations. They are curves which in some sense describe the shortest, or most efficient, path between two points in a space. 
More formally, in this article we will use the following definition of a geodesic: 
\begin{definition}\label{d:geo}
If $u,v$ are distinct points in a metric space $(X,d)$ then a \textit{geodesic} from $u$ to $v$ is a function $\gamma \colon [0,1] \to X$ such that $\gamma(0) = u$, $\gamma(1) = v$, and  
\begin{equation}\label{e:lengthmin}
 d(\gamma(s),\gamma(t)) = |s-t|d(u,v) \qquad \mbox{for all } s,t \in [0,1].
 \end{equation}
\end{definition} 
In this article we insist the length-minimising condition~\eqref{e:lengthmin} holds globally; some other authors call the curves defined in Definition~\ref{d:geo} \emph{minimising geodesics} or \emph{shortest paths}, and define geodesics as curves which satisfy~\eqref{e:lengthmin} only locally. 
 All the spaces $(X,d)$ in this article will be \textit{geodesic}, which means that for all distinct $u,v \in X$ there is a geodesic from $u$ to $v$. The celebrated Hopf--Rinow theorem implies that every connected Riemannian manifold that is a complete metric space is a geodesic space~\cite{b:hopf}. 
A space is said to be \textit{uniquely geodesic} if for all distinct $u,v \in X$ there is exactly one geodesic from $u$ to $v$. %
We make the following definition, which will be the main focus of this article. 
\begin{definition}
A metric space $(X,d)$ is a \emph{multigeodesic space} if for any distinct $u,v \in X$ there exist at least two distinct geodesics from $u$ to $v$. 
\end{definition}
It is not immediately obvious whether multigeodesic spaces exist. 
The circle is not an example because~\eqref{e:lengthmin} is a global condition. 
On first thought it may seem that the space $\mathbb{R}^2$ with the `taxicab' metric 
\[ d((x_1,y_1),(x_2,y_2)) = |x_1 - x_2| + |y_1 - y_2| \]
 provides an example, but it does not, since (for example) there will be only one geodesic between two different points with the same $y$-coordinate, as depicted in Figure~\ref{f:onenorm}. 
\begin{figure}[ht]
\center{\includegraphics[width=0.6\textwidth]
        {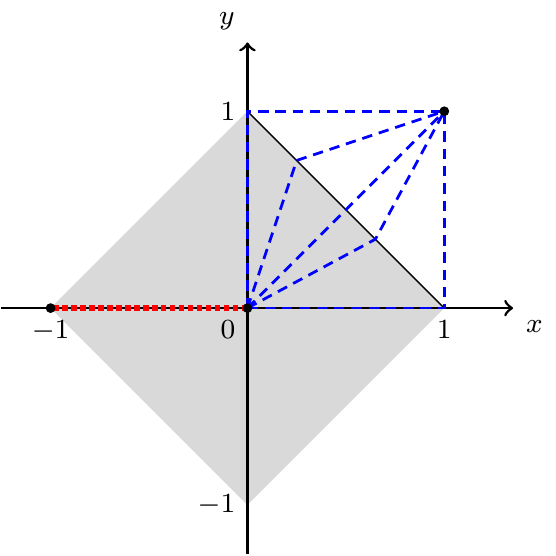}}
        \caption{\label{f:onenorm}
        The space $\mathbb{R}^2$ equipped with the taxicab metric (which is induced by the $1$-norm $||(x,y)||_1 = |x| + |y|$, whose unit ball is shaded grey). This space is geodesic, but neither uniquely geodesic nor multigeodesic. The dashed lines show several different geodesics between $(0,0)$ and $(1,1)$. The dotted line shows the unique geodesic between $(0,0)$ and $(-1,0)$. 
 }
\end{figure}

 One example of a multigeodesic space is what we call the~\emph{Laakso space}, a representation of which is shown in Figure~\ref{f:laakso}. The construction of this space is described in~\cite[Section~2]{b:lang}, which is in turn based on the construction in~\cite[Section~2]{b:laakso}. The space is built iteratively, like the Cantor set, except that instead of the middle part of each interval being removed, it is replaced by two `middle parts.' More precisely, each segment of length $L$ is replaced by six segments of length $L/4$, producing two geodesics between the endpoints of the initial segment. 
 Since this process is repeated at all scales, the Laakso space is multigeodesic. 
In Section~\ref{s:normed} we will provide a much wider variety of examples of multigeodesic spaces, including $L^1$ spaces, and give a characterisation of multigeodesic normed spaces. 
\begin{figure}[ht]
\center{\includegraphics[width=0.85\textwidth]
        {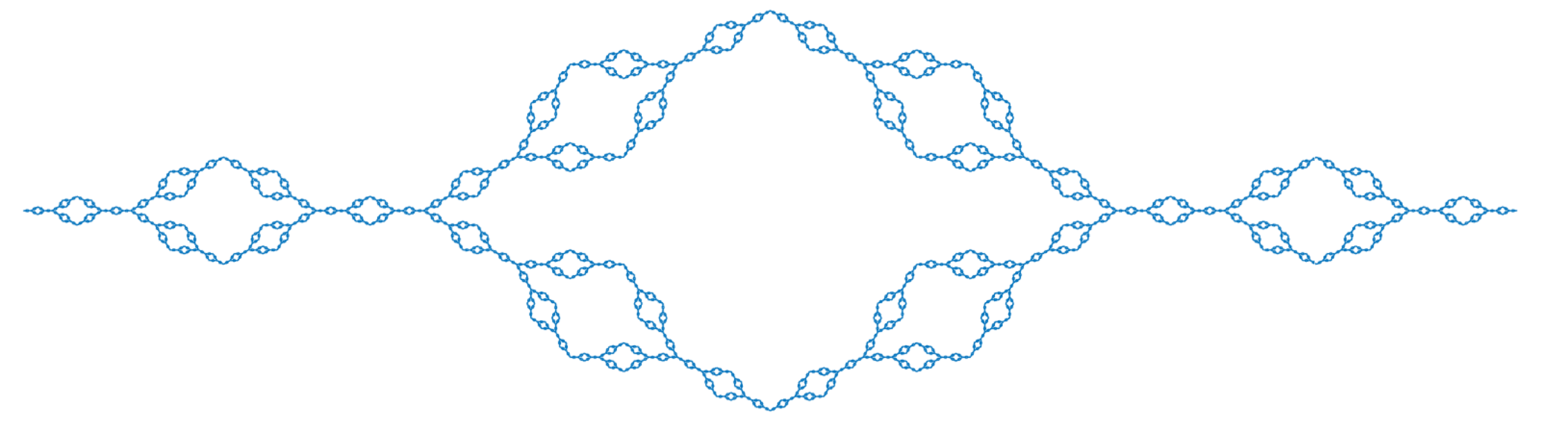}}
        \caption{\label{f:laakso}
        The Laakso space is multigeodesic. The metric here is different from the one induced by the Euclidean distance; indeed, it is shown in~\cite[Theorem~2.3]{b:lang} that the space cannot embed into Euclidean space via a bi-Lipschitz mapping. 
 }
\end{figure}

\section{Normed vector spaces}\label{s:normed}
In this section we will restrict our attention to normed vector spaces (which we assume will be real, with the metric induced by the norm). 
Any normed vector space is geodesic, since the straight line segment between any two distinct points is clearly a geodesic. It can be shown, as in~\cite[pages~5--6, 47--48]{b:convex}, that a normed vector space is uniquely geodesic if and only if the unit ball is strictly convex. Therefore, for example, any inner product space is uniquely geodesic, and for any $n \in \mathbb{N}$, the space $(\mathbb{R}^n,||\cdot||_p)$ is uniquely geodesic if $p \in (1,\infty)$ but not when $p=1$ or $\infty$. 
Theorem~\ref{thm:main} will provide a simple way of of characterising multigeodesic normed vector spaces, from which we can provide examples showing that such spaces exist. 
First, we need three lemmas. 
\begin{lemma}\label{lem:leqgeo}
  Let $(X,d)$ be a metric space, let $u,v \in X$, and suppose $\gamma \colon [0,1] \to X$ satisfies $\gamma(0) = u$, $\gamma(1) = v$. Then $\gamma$ is a geodesic if and only if for all $s,t \in [0,1]$ we have $d(\gamma(s),\gamma(t)) \leq |s-t|d(u,v)$.
\end{lemma}
\begin{proof}
 The forward implication is trivial, so for the backward implication suppose that for all $s,t \in [0,1]$ we have $d(\gamma(s),\gamma(t)) \leq |s-t|d(u,v)$. Then if $0\leq s\leq t \leq 1$, by the triangle equality,
  \begin{align*}d(u,v) = d(\gamma(0),\gamma(1)) &\leq d(\gamma(0),\gamma(s)) + d(\gamma(s),\gamma(t)) + d(\gamma(t),\gamma(1))\\ &\leq |s-0|d(u,v) + |t-s|d(u,v) + |1-t|d(u,v)\\ &= d(u,v),\end{align*}
  so there is equality throughout. In particular $d(\gamma(s),\gamma(t)) = |t-s|d(u,v)$, so $\gamma$ is a geodesic.  
\end{proof}
 Two geodesics are said to be \textit{disjoint} if their images do not intersect anywhere except possibly the endpoints. 
 We now construct disjoint geodesics between two points under the assumption that there are multiple distinct `intermediate' points between the two points. 
\begin{lemma}\label{lem:geo}
  Let $(X,||\cdot||)$ be a normed vector space and let $u,v \in X$. Suppose $x,y \in X$ are distinct and $C \in (0,1)$ is such that $||x-u|| = ||y-u|| = C||u-v||$ and $||v-x|| = ||v-y|| = (1-C)||u-v||$. Then the curves $\gamma$ and $\sigma$, defined below, are disjoint geodesics from $u$ to $v$: 
  \begin{align*}
   &\gamma(z) = \begin{cases} \frac{C-z}{C}u + \frac{z}{C}x & \mbox{ if } 0 \leq z \leq C\\
    \frac{1-z}{1-C}x + \frac{z-C}{1-C}v & \mbox{ if } C < z \leq 1, \end{cases} \\
&\sigma(z) = \begin{cases} \frac{C-z}{C}u + \frac{z}{C}y & \mbox{ if } 0 \leq z \leq C\\
    \frac{1-z}{1-C}y + \frac{z-C}{1-C}v & \mbox{ if } C < z \leq 1. \end{cases} 
    \end{align*}
\end{lemma}
\begin{proof}
Rescaling by $1/||u-v||$ and then translating, we may assume without loss of generality that $u=0$ and $||u-v||=1$, so $||x|| = C$ and $||v-x|| = 1-C$. Let $0\leq s <t \leq 1$. If $t\leq C$ then $||\gamma(t)-\gamma(s)|| = ||(t-s)x/C|| = t-s$. If $C \leq s$ then $||\gamma(t)-\gamma(s)|| = ||(t-s)(v-x)/(1-C)||= t-s$. If $s<C<t$ then 
\[ ||\gamma(t)-\gamma(s)|| \leq ||\gamma(t)-\gamma(C)|| + ||\gamma(C)-\gamma(s)|| = t-C+C-s = t-s.\] 
Therefore by Lemma~\ref{lem:leqgeo}, $\gamma$ is a geodesic. By a similar argument, $\sigma$ is also a geodesic.

  Now suppose that $w,z \in [0,1]$ are such that $\gamma(w) = \sigma(z)$. Then $||\gamma(w)-u|| = ||\sigma(z)-u||$, i.e. $w||v-u||=z||v-u||$, so $w=z$. If $z\leq C$ then $\gamma(w) = \sigma(z)$ implies that $zx/C=zy/C$, but $x\neq y$, so $z=0$. If $z>C$ then ${(1-z)x}/{(1-C)} = {(1-z)y}/{(1-C)}$, so $z=1$. Hence $\gamma$ and $\sigma$ intersect only at $u$ and $v$, so are disjoint. 
\end{proof}
We can now give a condition that guarantees the existence of multiple geodesics between two points. 
\begin{lemma}\label{l:main}
  Let $(X,||\cdot||)$ be a normed vector space, and let $u,v \in X$ be distinct. The following are equivalent: 
  \begin{enumerate}
  \item\label{i:distinct} There are at least two distinct geodesics from $u$ to $v$. 
  \item\label{i:uncountable} There is an uncountable family of pairwise-disjoint geodesics from $u$ to $v$. 
  \item\label{i:existc} There exist $C \in (0,1)$ and distinct $x,y \in X$ such that 
  \begin{align}\label{e:easycondition} 
  \begin{split}
  &||x-u|| = ||y-u|| = C||u-v||, \mbox{ and} \\
  &||v-x|| = ||v-y|| = (1-C)||u-v||.
  \end{split}
  \end{align}
  \end{enumerate}
  \end{lemma}
\begin{proof}
Rescaling, we assume without loss of generality that $||u-v||=1$. 
The implication~\ref{i:uncountable}~$\Rightarrow$ \ref{i:distinct}~is trivial. To show \ref{i:distinct}~$\Rightarrow$~\ref{i:existc}, note that if $\gamma$ and $\sigma$ are distinct geodesics from $u$ to $v$ then there exists $C \in (0,1)$ such that $\gamma(C) \neq \sigma(C)$, so \ref{i:existc} holds if we set $x = \gamma(C)$ and $y = \sigma(C)$. 

It remains to show~\ref{i:existc}~$\Rightarrow$~\ref{i:uncountable}. Assume~\ref{i:existc} holds, and for $\lambda \in [0,1]$ define $f(\lambda) \coloneqq \lambda x + (1-\lambda) y$. Then 
\[ ||f(\lambda) - u|| \leq \lambda||x-u|| + (1-\lambda)||y-u|| = C. \]
Similarly $||v-f(\lambda)|| \leq 1-C$, so again using the triangle inequality, $||f(\lambda) - u|| = C$ and $||v-f(\lambda)|| = 1-C$. By Lemma~\ref{lem:geo} there exists an uncountable family $\{\gamma_\lambda\}_{0 \leq \lambda \leq 1}$ of pairwise-disjoint geodesics from $u$ to $v$, with $\gamma_\lambda(C) = f(\lambda)$. 
\end{proof}
The dashed lines in Figure~\ref{f:onenorm} show a selection from the uncountable family of geodesics from $(0,0)$ to $(1,1)$ in the space $(\mathbb{R}^2,||\cdot ||_1)$. Each is given by a straight line from $(0,0)$ to an intermediate point $(x,1-x)$ for some $x \in [0,1]$, followed by the straight line from $(x,1-x)$ to $(1,1)$.

Theorem~\ref{thm:main} now follows from Lemma~\ref{l:main}. 
Condition~\ref{i2:existc} below is often straightforward to check, and provides a useful characterisation of multigeodesic normed spaces. 
\begin{theorem}\label{thm:main}%
If $(X,||\cdot ||)$ is a normed vector space, the following are equivalent: 
\begin{enumerate}
  \item\label{i2:distinct} The space $(X,||\cdot ||)$ is multigeodesic. 
  \item\label{i2:uncountable} Between any two distinct points in $X$ there is an uncountable family of pairwise-disjoint geodesics. 
  \item\label{i2:existc} For all $x \in X$ with $||x|| = 1$ there exist $C \in (0,1)$ and $y \in X \setminus \{Cx\}$ such that $||y|| = C$ and $||x-y|| = 1-C$. 
  \end{enumerate}
\end{theorem}
\begin{proof}
The implication~\ref{i2:uncountable}~$\Rightarrow$~\ref{i2:distinct} is trivial. 
The implication~\ref{i2:distinct}~$\Rightarrow$~\ref{i2:existc} follows from the corresponding implication in Lemma~\ref{l:main} with $u=0$ and $v=x$. 
For~\ref{i2:existc}~$\Rightarrow$~\ref{i2:uncountable}, note that if $u,v \in X$ are distinct, then since condition~\ref{i2:existc} with $x = \frac{v-u}{||v-u||}$ holds, condition~\ref{i:existc} from Lemma~\ref{l:main} also holds, so there is an uncountable family of pairwise-disjoint geodesics from $u$ to $v$, as required. 
\end{proof}

We now give several corollaries. 
We provide a simple first example of a multigeodesic normed space. 
\begin{corollary}\label{cor:cts}
The normed vector space $(C[0,1], ||\cdot||_1)$ is multigeodesic. 
\end{corollary}
\begin{proof}
Let $g \colon [0,1] \to \mathbb{R}$ be a continuous function with $||g||_1 = 1$. Define $h(x) = xg(x)$ and let $C \coloneqq ||h||_1$. Note that $0 \leq |h| \leq |g|$ pointwise, and there is some interval on which $0<|h|$ and some interval on which $|g(x)|(1-x)>0$, so $|h|<|g|$, so $0 < C < 1$. Now,
\begin{align*}
    ||h||_1+||g-h||_1 &= \int_0^1|xg(x)|dx + \int_0^1|g(x)-xg(x)|dx\\
    &= \int_0^1(x+(1-x))|g(x)|dx\\
    &= ||g||_1,
  \end{align*}
so $||g-h||_1=||g||_1-||h||_1=1-C$. Moreover, there is some interval of $x$ on which $g(x)(C-x)$ is non-zero, so $Cg(x) \neq h(x)$. Therefore by condition~\ref{i2:existc} of Theorem~\ref{thm:main}, $(C[0,1], ||\cdot||_1)$ is a multigeodesic space. 
\end{proof}
\begin{figure}[ht]
\center{\includegraphics[width=.85\textwidth]
        {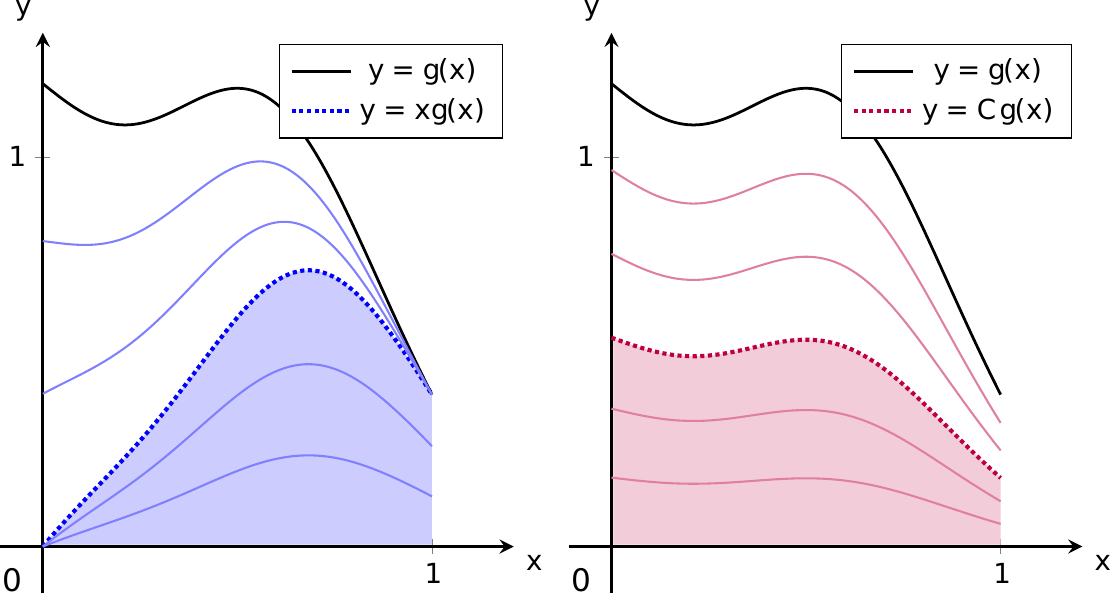}}
        \caption{\label{f:cts}
        Some functions lying on two disjoint geodesics between 0 and a function $g$ with $||g||_1 = 1$, in the space $(C[0,1], ||\cdot||_1)$ from Corollary~\ref{cor:cts}. Left: the geodesic goes through the function $h(x) = xg(x)$. Right: the geodesic goes through the alternative `intermediate' function $Cg$, where $C = ||h||_1 \approx 0.45$ is such that the two shaded regions have equal areas. 
}
\end{figure}
Examples of the functions $g$, $Cg$ and $h$ from the proof of Corollary~\ref{cor:cts} are shown in Figure~\ref{f:cts}. Notice that the example in Corollary~\ref{cor:cts} was an infinite-dimensional vector space. 
\begin{corollary}\label{c:finite}
No finite-dimensional normed vector space is multigeodesic. 
\end{corollary}
\begin{proof}
Every finite-dimensional vector space is isometrically isomorphic to $\mathbb{R}^n$ equipped with some norm $||\cdot||$, so this is the space that we will work with. Note that the closed unit ball $B$ (with respect to $||\cdot||$) is compact. %
The Euclidean norm $||\cdot||_2$ is continuous on $\mathbb{R}^n$, so it attains a maximum on $B$, say $||x||_2$ for some $x \in B$. Clearly $||x||=1$ (since $||x/||x|| \, ||_2 \geq ||x||_2$, there must be equality). 

Let $C \in (0,1)$ and suppose that $y \in \mathbb{R}^n$ satisfies $||y|| = C$ and $||y-x|| = 1-C$. Then by the definition of $x$, $y$ lies in the intersection of two closed Euclidean balls, namely the ball centred at the origin with radius $C||x||_2$, and the ball centred at $x$ with radius $(1-C)||x||_2$. But these balls intersect only in the single point $Cx$, so $y = Cx$. Therefore condition~\ref{i2:existc} of Lemma~\ref{l:main} is not satisfied, so there is only one geodesic from $0$ to $x$. 
\end{proof}
For the space $(\mathbb{R}^2,||\cdot||_1)$, for example, the point $x$ in the proof of Corollary~\ref{c:finite} could be $(-1,0)$, shown in Figure~\ref{f:onenorm}. 
The following corollary demonstrates that while some infinite-dimensional normed spaces are multigeodesic, others are not. 
\begin{corollary}\label{cor:lp}
Let $n \in \mathbb{N}$ and suppose $K \subseteq \mathbb{R}^n$ is Lebesgue measurable with $n$-dimensional Lebesgue measure $\mu(K) = 1$. %
Then $L^p(K)$ is multigeodesic if $p=1$ but not if $p \in (1,\infty]$. 
\end{corollary}
Here, $L^p(K)$ denotes the usual Lebesgue space with respect to the $\sigma$-algebra of Lebesgue measurable subsets of $K$ and Lebesgue measure $\mu$. Note that these spaces are complete, unlike the space $(C[0,1], ||\cdot||_1)$ from Corollary~\ref{cor:cts}. 
\begin{proof}
We begin with the $p=1$ case. Given $g \in L^1(K)$ with $||g||_1 = 1$, it is a standard exercise to show that there exists $E \subset K$ such that 
\[ 0 < \int_E |g| d\mu < \int_K |g| d\mu.\] %
Equivalently, $0 < C < 1$, where $C \coloneqq ||g \cdot \chi_E||_1$, where $\chi_E$ denotes the indicator function of $E$. Then condition~\ref{i2:existc} of Theorem~\ref{thm:main} is satisfied by the function $g \cdot \chi_E$, which is distinct from $C \cdot g$ in $L^1(K)$. Therefore $L^1(K)$ is multigeodesic. 

We now suppose $1 < p \leq \infty$. By the triangle inequality and H\"older's inequality, for all $f \in L^1(K)$, 
\[ 1 = ||\chi_K||_1 \leq ||f \cdot \chi_K||_1 + ||(\chi_K - f) \cdot \chi_K||_1 \leq ||f||_p + ||\chi_K - f||_p,\] 
with equality only if $f$ is almost surely constant. %
Therefore condition~\ref{i:existc} of Lemma~\ref{l:main} is not satisfied, so there is only one geodesic from $0$ to $\chi_K$. 
\end{proof}

\section{Metric spaces}\label{s:metric}
It is now natural to consider whether similar conclusions can be drawn in general metric spaces. 
It turns out that the situation in this setting is more complicated. Indeed, there exist multigeodesic metric spaces which have distinct points between which no two geodesics are disjoint, in contrast to condition~\ref{i2:uncountable} of Theorem~\ref{thm:main} for normed spaces. This is clearly the case for the Laakso space, shown in Figure~\ref{f:laakso}.  
In fact, spaces with this property can be constructed from any other multigeodesic space $(Y,d_Y)$ by gluing two copies of it together at a point $y_0 \in Y$, as we now describe. 
Define $X \coloneqq (Y \times \{0,1\})/\sim$ where $\sim$ is the equivalence relation whose only non-trivial relation is $(y_0,0) \sim (y_0,1)$. 
We equip $X$ with the metric given by 
\[ d((x,i),(y,i)) = d_Y(x,y) \qquad \mbox{and} \qquad d((x,i),(y,1-i)) = d(x,y_0) + d(y,y_0)\] 
 for $i \in \{0,1\}$ and $x,y \in Y$. 
If $\gamma$ is a geodesic in $Y$ then the functions $t \mapsto (\gamma(t),0)$ and $t \mapsto (\gamma(t),1)$ clearly define geodesics in $X$. Therefore for any $x,y \in Y$ there are multiple distinct geodesics between $(x,0)$ and $(y,0)$, and between $(x,1)$ and $(y,1)$. 
If $x,y \in Y \setminus \{y_0\}$ then there are many different geodesics between $(x,0)$ and $(y,1)$ formed by concatenating any geodesic from $(x,0)$ to $(y_0,0)$ with any geodesic from $(y_0,1)$ to $(y,1)$. 
Therefore $X$ is indeed a multigeodesic space. 
However, any two geodesics which both start in the open set $(Y \times \{0\}) \setminus \{(y_0,0)\}$ and end in the open set $(Y \times \{1\}) \setminus \{(y_0,1)\}$ must pass through the only other point in $X$ (namely $(y_0,0)$) by continuity, so they cannot be disjoint.

We now observe that a portion of one geodesic can be replaced with another geodesic to form a different geodesic. 
\begin{lemma}\label{lem:paste}
  Let $(X,d)$ be a multigeodesic metric space, let $\gamma \colon [0,1] \to X$ be a geodesic, let $0 \leq s<t \leq 1$, and let $\sigma \colon [0,1] \to X$ be a geodesic from $\gamma(s)$ to $\gamma(t)$. Then $\gamma^{(\sigma,s,t)} \colon [0,1] \to X$ defined by
  \[
    \gamma^{(\sigma,s,t)}(x) =
  \begin{cases}%
    \gamma(x) & \mbox{if } 0\leq x \leq s \mbox{ or } t\leq x \leq 1 \\
    \sigma(\frac{x-s}{t-s}) & \mbox{if } s<x<t
  \end{cases}
  \]
  is a geodesic from $\gamma(0)$ to $\gamma(1)$.
\end{lemma}
\begin{proof}
It is a straightforward calculation to verify that the condition~\eqref{e:lengthmin} is satisfied. 
\end{proof}
We can now construct many different geodesics between any two points. What makes the proof of Lemma~\ref{l:branch} non-trivial is that in light of the discussion at the start of this section, we cannot always guarantee that the portion of the geodesic which we replace is disjoint from the original geodesic. 
\begin{lemma}\label{l:branch}
Let $(X,d)$ be a multigeodesic metric space, let $\gamma \colon [0,1] \to X$ be a geodesic, and suppose $0 < t < 1$. Then there exists a geodesic $\tau$ from $\gamma(0)$ to $\gamma(1)$ satisfying  
\[ \inf\{ \, s \in (0,1) : \tau(s) \neq \gamma(s) \, \} = t.\] 
\end{lemma}
\begin{proof}
Let $n \in \mathbb{N}$ be large enough that $t + 2^{-n} < 1$. Since $X$ is multigeodesic, we can use the Axiom of Choice to choose, for each $m > n$, a geodesic $\sigma_m$ from ${\gamma(t+2^{-m})}$ to ${\gamma(t+2^{-(m-1)})}$ such that $\sigma_m((0,1)) \neq \gamma((t+2^{-m},t+2^{-(m-1)}))$. Define ${\gamma_n \coloneqq \gamma}$. For $m > n$, using notation from Lemma~\ref{lem:paste}, inductively define 
\[ \gamma_m \coloneqq \gamma_{m-1}^{(\sigma_m,t+2^{-m},t+2^{-(m-1)})}.\] 
 Then $\gamma_m(s) = \gamma(s)$ whenever $s \leq t+2^{-m}$ or $s \geq t+2^{-n}$, but there exists $s \in (t+2^{-m},t+2^{-(m-1)})$ such that $\gamma_m(s) \neq \gamma(s)$. For all $s \in [0,1]$, the sequence $(\gamma_m(s))_{m=1}^{\infty}$ is eventually constant, so define $\tau$ by \[\tau(s) \coloneqq \lim_{m \to \infty} \gamma_m(s).\] 
 Since each $\gamma_m$ is a geodesic, $\tau$ is also a geodesic. %
 Clearly $\tau(0) = \gamma(0)$, $\tau(1) = \gamma(1)$, and $\inf\{ \, s \in (0,1) : \tau(s) \neq \gamma(s) \, \} = t$, as required. 
\end{proof}
As with normed spaces, we now have the following result for arbitrary metric spaces. 
\begin{prop}\label{prop:uncountable}
  In any multigeodesic metric space there are uncountably many pairwise-distinct geodesics between any two distinct points. 
\end{prop}
\begin{proof}
This is immediate from Lemma~\ref{l:branch}. 
\end{proof}
A geodesic $\gamma \colon [0,1] \to X$ in a metric space $X$ is said to \emph{branch} at time $t \in (0,1)$ if there exists a geodesic $\sigma \colon [0,1] \to X$ such that 
\[ \inf \{ \, s \in (0,1) : \sigma(s) \neq \gamma(s) \, \} = t.\]
The question of when geodesics branch has received some attention, for instance in~\cite{b:mietton,b:ohta}. 
\begin{prop}
In any multigeodesic metric space, every geodesic branches at every time $t \in (0,1)$. 
\end{prop}
\begin{proof}
This is immediate from Lemma~\ref{l:branch}. 
\end{proof}

\section*{Acknowledgements}
The author thanks Jonathan Fraser for asking the question which stimulated this work, and for helpful discussions. He thanks Murad Banaji, Kenneth Falconer, Alex Rutar and Yuchen Zhu for useful comments on a draft version of this article. The author thanks two anonymous referees for comments which improved the quality of the article. He is especially grateful to a referee for producing Figure~\ref{f:laakso} and for making him aware of references~\cite{b:laakso,b:lang}. This work was supported by the Leverhulme Trust under grant RPG-2019-034.

\section*{References}
\printbibliography[heading=none]%

\end{document}